\documentclass[english]{article}
\usepackage{amsthm}
\usepackage{amsmath}
\usepackage{amssymb}

\makeatletter
\newtheorem{thm}{Theorem}[section]
  \newtheorem{cor}[thm]{Corollary}
  \newtheorem{prop}[thm]{Proposition}
 \ifx\proof\undefined\
   \newenvironment{proof}[1][\proofname]{\par
     \normalfont\topsep6\p@\@plus6\p@\relax
     \trivlist
     \itemindent\parindent
     \item[\hskip\labelsep
           \scshape
       #1]\ignorespaces
   }{%
     \endtrivlist\@endpefalse
   }
   \providecommand{\proofname}{Proof}
 \fi

\usepackage{mathrsfs}\usepackage{amsthm}\usepackage{amscd}

\usepackage{hyperref}
\usepackage{graphics}
\usepackage{a4wide}
\@ifundefined{definecolor}
 {\usepackage{color}}{}

\newcommand{\fdot}{\skew{6}{\dot}{f}}
\newcommand{\gdot}{\skew{4}{\dot}{g}}

\newcounter{EQNR}
\renewcommand{\contentsname}{Contents\\
{\footnotesize\normalfont(A table of contents should normally not
be included)}}

\makeatother

\usepackage{babel}

\begin{document}

\title{Heat kernels on regular graphs and generalized Ihara zeta function
formulas}

\date{G. Chinta, J. Jorgenson, and A. Karlsson%
\thanks{The first and second authors acknowledge support provided by grants
from the National Science Foundation and the Professional Staff Congress
of the City University of New York. The third author received support
from SNSF grant 200021\_132528/1. Support from Institut Mittag-Leffler
(Djursholm, Sweden) is also gratefully acknowledged by all authors. We thank Pierre de la Harpe for a number of corrections.%
}}
\maketitle
\begin{abstract}
\noindent
We establish a new formula for the heat kernel on regular trees in
terms of classical $I$-Bessel functions. Although the formula is
explicit, and a proof is given through direct computation, we also
provide a conceptual viewpoint using the horocyclic transform on
regular trees.  From periodization, we then obtain a heat kernel
expression on any regular graph. From spectral theory, one has another
expression for the heat kernel as an integral transform of the
spectral measure.  By equating these two formulas and taking a certain
integral transform, we obtain several generalized versions of the
determinant formula for the Ihara zeta function associated to finite
or infinite regular graphs.  Our approach to the Ihara zeta function
and determinant formula through heat kernel analysis follows a similar
methodology which exists for quotients of rank one symmetric spaces.
\end{abstract}

\section{Introduction}

Let $X$ be a $(q+1)$-regular graph for an integer $q>0$. There is
an associated heat kernel $K_{X}(t,x_{0},x)$ corresponding to the
Laplacian formed by considering the adjacency matrix on $X$.
We show that in a natural way the building blocks of $K_{X}$ are the functions
\[
q^{-r/2}e^{-(q+1)t}I_{r}(2\sqrt{q}t),
\]
where $r\in\mathbb{Z}_{\geq0}$ , time $t\in\mathbb{R}_{\geq0}$, and
$I_{r}$ is the classical $I$-Bessel function of order $r$.  The
expression for the heat kernel on $X$ comes from a new formula for the
heat kernel on regular trees (Proposition \ref{pro:The-heat-kernel})
which we prove in this article.  Our expression is quite different
from a previous formula due to F. Chung and S.-T. Yau
\cite{ChungYau99}, which we describe in subsection
\ref{sub:heatontree}, Equations (\ref{eq:CYa}, \ref{eq:CYb}).  If we
write the functions in the above stated building block as
\[
q^{-r/2}\cdot e^{-(\sqrt{q}-1)^{2}t}\cdot e^{-2\sqrt{q}t}I_{r}(2\sqrt{q}t),
\]
then there is moreover a near-perfect analogy with the building blocks
of typical heat kernel expressions on Riemannian symmetric spaces, which
have the form
\[
F(r)\cdot e^{-at}\cdot\frac{1}{\sqrt{4\pi t}^{d}}e^{-r^{2}/4t}
\]
for certain constants $a,d$ and function $F$ that allow further
interpretations; we refer to the survey article \cite{JoLa01} and the
references therein for further discussion.

More precisely, we prove the following result.

\begin{thm}
\label{thm:The-heat-kernel}The heat kernel on a $(q+1)$-regular
graph $X$ is given by
\[
K_{X}(t,x_{0},x)=e^{-(q+1)t}\sum_{m=0}^{\infty}b_{m}(x)q^{-m/2}I_{m}(2\sqrt{q}t),
\]
where $I_{m}$ is the I-Bessel function of order $m$,
$b_{m}(x)=c_{m}(x)-(q-1)(c_{m-2}(x)+c_{m-4}(x)+...)$ and $c_{m}(x)$ is
the number of geodesics from a fixed base point $x_{0}$ to $x$ of
length $m\geq 0$.
\end{thm}

To be specific, we define a \emph{geodesic} in a graph to be a path
without back-tracking.  The terminology is consistent with concepts
from Riemannian geometry where a geodesic is a path which locally is
distance minimizing. Moreover, a \emph{closed geodesic} is a closed
path without back-tracking or tails. We defer until section
\ref{sec:Preliminaries} for more details on these definitions and for
a precise axiomatic characterization of the heat kernel. In view of a
combinatorial observation, Proposition \ref{pro:pathcombinatorics}, we
may formulate the following result.

\begin{cor}\label{corollary}
  In addition to the stated assumptions, suppose that $X$ is vertex
  transitive. Let $N_{m}^{0}$ denote the number of closed geodesics of
  length $m$ in $X$ with base point $x_{0}$. Then
\[
K_{X}(t,x_{0},x_{0})=K_{q+1}(t,x_{0},x_{0})+e^{-(q+1)t}\sum_{m=1}^{\infty}
N_{m}^{0}q^{-m/2}I_{m}(2\sqrt{q}t)
\]
where $K_{q+1}$ denotes the heat kernel of the $(q+1)$-regular tree.
\end{cor}

For finite, not necessarily vertex transitive, graphs $X$, similar
formulas were previously proved by Ahumada \cite{Ahumada87},
Terras-Wallace \cite{Terras03} and Mn\"ev \cite{Mnev07}. Note that our
formula holds for infinite graphs as well, and is therefore more
general than the finite graph case.
With our methods,
one can also deduce the formula for the finite non-vertex transitive
case using Proposition \ref{pro:combforfinitegraphs}.

There is a second expression for the heat kernel coming from spectral
considerations. Equating the two expressions for the heat kernel, as
in known approaches to the Poisson summation formula or the Selberg
trace formula, one obtains an identity which is a type of theta
inversion formula.  From the identity, we will apply a certain
integral transform, which amounts to a Laplace transform with a change
of variables from which we obtain the logarithmic derivative of the
Ihara zeta function.  This procedure is motivated by McKean
\cite{McKean} in his approach to the Selberg zeta function and was
axiomatized in \cite{JoLa01} to abstract settings.  In the end, one
obtains determinantal formulas for Ihara zeta-like functions.  In a
special case, we recover the standard formula stemming from Ihara's
work \cite{Ihara66}, which in turn is generalized in
\cite{Hashimoto89,Bass92,FoZe99,KoSu00,ST} for finite graphs; see
subsection \ref{sub:Ihara's-determinantal-formula}.

We now describe one sample outcome which comes from the above
described sequence of calculations. Let $X$ be a vertex transitive
$(q+1)$-regular graph. We define the associated \emph{Ihara zeta
  function} of $X$ by
\[
\zeta_{X}(u)=
\exp\left\{ \sum_{m=1}^{\infty}\frac{N_{m}^{0}}{m}u^{m}\right\} ,
\]
where $N_{m}^{0}$ is the number of closed geodesics of length $m$
starting at a fixed vertex $x_{0}$. For finite graphs, the classical Ihara zeta function is just our Ihara zeta function raised to the power equaling the number of vertices.

\begin{thm}
\label{thm:iharainfinite}Let $X$ be a vertex transitive $(q+1)$-regular
graph with spectral measure $\mu$ for the Laplacian. Then
\[
\zeta_{X}(u)^{-1}=(1-u^{2})^{(q-1)/2}\exp\left(\int\log\left(1-(q+1-\lambda)u+qu^{2}\right)d\mu(\lambda)\right).
\]
\end{thm}

Again, we defer to section \ref{sec:Preliminaries} for the definitions
of the Laplacian and spectral measure. There are some papers in the
literature defining Ihara zeta functions for infinite graphs, in particular
Clair and Mokhtari-Sharghi \cite{Bryan02}, Grigorchuk-Zuk \cite{Grigorchuk04}
and Guido, Isola and Lapidus \cite{Guido08}. Their definitions are
at least a priori somewhat different in that they typically look at
approximations by finite graphs and also using von Neumann trace for
group operator algebras. In many cases it coincides with our definition
and thereby the formula in Theorem \ref{thm:iharainfinite} can be
recovered in those references.

A number of interesting examples of
Ihara zeta functions for infinite Cayley graphs can be found in \cite{Grigorchuk04}.
Additionally, we refer to the articles \cite{Sun1}, \cite{Sun2}
and \cite{Sun3} in which the author provides a fascinating discussion
in which the spectral and zeta function analysis on graphs is compared
to similar studies in spectral theory on symmetric spaces and zeta functions
from number theory.

In summary, we have given a new expression for the heat kernel associated to any
regular graph.  One can quickly deduce the Ihara determinant formula
and a number of interesting extensions, not the least of which is
to infinite transitive graphs. As is well-known, one main application
of such formulas is to the study of counting closed geodesics. For
us, it is also significant that our analysis provides yet another
instance of when the heat kernel yields zeta functions together with
their main functional relation, just as in the case of Riemann, Selberg,
and beyond. In particular, the present paper can be viewed in the
context of the last section of \cite{JoLa01}.

\section{Preliminaries\label{sec:Preliminaries}}

\subsection{Graphs}

We follow the definitions in Serre's book \cite{Serre80}. A \emph{graph}
$X$ consists of a set $VX$ which are called \emph{vertices}, a set $EX$
which are called \emph{edges}, and two maps
\[
EX\rightarrow VX\times VX,\text{ }y\mapsto(o(y),t(y))
\]
and
\[
EX\rightarrow EX,\text{ }y\mapsto\overline{y}
\]
such that for each $y\in EX$ we have that
$\overline{\overline{y}}=y,$ $\overline{y}\neq y$ and $o(y)=t(\overline{y})$.

The vertices $o(y)$ and $t(y)$ are the \emph{extremities} of the
edge $y$. Two vertices are \emph{adjacent} if they are extremities
of an edge. The \emph{degree} of a vertex $x$ is
\[
\deg x=Card\{y\in EX:o(y)=x\}.
\]
A graph is $d$-\emph{regular} if each vertex has degree $d$.

There is an obvious notion of morphism. Let $PATH_{n}$ denote the
graph with vertices ${0,1,2,...,n}$ and (half of) the edges
are given by $[i,i+1]$, $i=0,...,n-1$. A \emph{path} (of length
$n)$ is a morphism $c$ from $PATH_{n}$ into the graph. The sequences
of edges $y_{i}=c([i,i+1])$ (such that $t(y_{i})=o(y_{i+1})$) determines
the path. In particular a path is oriented. There is a \emph{backtracking}
if for some $i$ that $y_{i+1}=\overline{y_{i}}$ and there is a \emph{tail}
if $y_{0}=\overline{y}_{n-1}.$ A path is \emph{closed }if $c(0)=c(n)$$.$
A \emph{geodesic} is a path without backtracking. A \emph{geodesic
loop} (or \emph{circuit }in Serre's terminology) is a closed path
that is a geodesic. A \emph{closed geodesic} is a closed path with
no tail and without backtracking. (This is analogy with Riemannian
geometry where a closed geodesic, as opposed to a geodesic loop, is
required to be smooth also at the start/end point). 
The path of length zero counts as a closed geodesic and, therefore, is a geodesic loop.
Additionally, every closed path with one edge counts as a closed geodesic.  Any length
two geodesic loop is also a closed geodesic, but the closed path $y .\overline{y}$ is neither.  

A \emph{prime geodesic} is an equivalence class of closed geodesics $\left[c\right]$,
where the equivalence class is forgetting the starting point and which
is primitive in the sense that it is not a power of another closed
geodesic. The latter means by definition that there is no closed geodesic
$d$ and integer $n>1$ such that $\left[c\right]=\left[d^{n}\right]$,
which says in words that $c$ is not just a geodesic that traverses
another one $n$ number of times. $ $

An \emph{orientation} is a subset $EX_{+}$ of edges such that $EX$
is the disjoint union of $EX_{+}$ and $\overline{EX_{+}}.$ With
the data of a graph one can associate a geometric realization: start
with the discrete topology,\ take $VX\times\lbrack0,1]$ and make
identification based on the maps $o$ and $t.$

A \emph{tree} is a connected nonempty graph without geodesic loops.

We will in particular consider vertex transitive graphs, that means
that there is a group of automorphisms which is transitive on the
vertices. In particular such a graph is of course regular. A rich
source of such graphs is provided by \emph{Cayley graphs} of groups:
Let $G$ be a group and let $S$ be a subset of $G.$ We denote by
$X(G,S)$ the oriented graph having $G$ as vertices and $EX_{+}=G\times S$
with $o(g,s)=g$ and $t(g,s)=gs$ for each edge $(g,s).$

Let $X$ be a graph on which a group $G$ acts. An \emph{inversion}
is a pair consisting of an element $g$ and an edge $y$ such that
$gy=\overline{y}.$ If $G$ acts without inversions (which is the
same as saying that there is an orientation of $X$ preserved by $G$)
we can define the quotient graph $G\backslash X$ in an obvious way;\ the
respective edge and vertex sets are the corresponding quotients. (To
get rid of inversions one may pass to a barycentric division.) As in topology we
say that $X$ is a \emph{regular covering} of $Y$ if there is a group which
acts on $X$ freely and without inversion with quotient $Y$.

Let $X$ be a $(q+1)$-regular graph. Then its universal covering
is the $(q+1)$-regular tree, and the covering group acts freely on
the tree without inversion. The covering group is a free group.

\subsection{Path counting in graphs}

We fix a base vertex $x_{0}$ in a graph $X$ and define the following counting
functions which will be used in this paper:
\begin{itemize}
\item $a_{k}(x)$ is the number of paths of length $k$ from $x_{0}$ to
$x$,
\item $c_{k}(x)$ is the number of geodesics of length $k$ from $x_{0}$
to $x$,
\item $c_{k}^{0}=c_{k}(x_{0})$ is the number of geodesic loops of length $k$
starting at $x_{0}$,
\item $c_k$ the number of geodesic loops of length $k$, from some starting
point with a distinct direction,
\item $N_{k}^0$ is the number of closed geodesics of length $k$ starting
at $x_{0}$,
\item $N_k$ is the number of closed geodesics of length $k$, from some starting
point with a distinct direction,
\item $\pi_k$ is the number of prime geodesics of length $k$.
\end{itemize}

The sequences $\{c_k\},$ $\{N_k\}$ and $\{\pi_k\}$ only make sense for
finite graphs, since if the graph is infinite, these values are
typically infinite.  For finite vertex transitive graphs, the
sequences $\{N_{k}^0\}$ and $\{N_k\}$ are related by the number of
vertices, i.e. starting points. Specifically, if the graph $X$ has $n$
vertices, then $N_{k}^0\cdot n = N_k$ for all $k$.  Also, $N_k$ and
$\pi_k$ have a precise relationship, see e.g. \cite{Terras11} or
\cite{Grigorchuk04}.  Finally, we recall the conventions that
$c_{0}^0=N_{0}^0=1$, $c_{1}^0=N_{1}^{0}$ and $c_{2}^0=N_{2}^{0}$.

\begin{prop}
\label{pro:pathcombinatorics}Let $X$ be a transitive $(q+1)$-regular
graph. Then for $k \geq 3$, following relation holds true
\[
N_{k}^0=c_{k}^0-(q-1)(c_{k-2}^0+c_{k-4}^0+...)
\]
the last term being $c_{1}^0$ or $c_{2}^0$ depending on the
parity of $k.$ \end{prop}
\begin{proof}
  This is similar to an argument in \cite{Serre97}. A geodesic loop of
  length $k \geq 3$ which is not a closed geodesic has the form
  $y_{1}.z.\overline{y}_{1}$ where $z$ is a geodesic loop of length
  $k-2$. There are two possibilities, either $z$ is a closed geodesic
  or not. If we fix $z$, then the number of possibilities for $y_{1}$
  is $q-1$ in the first case and $q$ in the second case.  Since $X$ is
  vertex transitive, we may freely change the starting point of any
  loop, namely $z$.  With this in mind, we obtain the recursive
  relation that
\[
c_{k}^0-N_{k}^0=(q-1)N_{k-2}^0+q(c_{k-2}^0-N_{k-2}^0)=
(c_{k-2}^0-N_{k-2}^0)+(q-1)c_{k-2}^0,
\]
which we can write as
$$
N_{k}^0 - N_{k-2}^0 = c_{k}^0 - qc_{k-2}^0.
$$
Using that $c_{1}^0=N_{1}^{0}$ and $c_{2}^0=N_{2}^{0}$, the proposition follows by
induction on $k$.
\end{proof}

With a proof similar to the one given in the above proposition, we obtain
the following result.

\begin{prop}
\label{pro:combforfinitegraphs}Let $X$ be a finite $(q+1)$-regular
graph. Then for $k \geq 3$, the following relation holds true
\[
N_k=c_k-(q-1)(c_{k-2}+c_{k-4}+...)\]
the last term being $c_{1}$ or $c_{2}$ depending on the
parity of $k$.
\end{prop}

\subsection{The combinatorial Laplacian and heat kernel\label{sub:The-combinatorial-Laplacian}}

Given a $(q+1)$-regular graph $X$ and function $f$ on the vertices
$X$, the Laplacian of $f$, written as $\Delta f$, is the function of
the vertices of $X$ which is defined by the formula
\[
\Delta f(x)=(q+1)f(x)\ \ -\sum_{e\text{ s.t. }o(e)=x}f(t(e)).
\]

The Laplacian is a semi-positive, bounded self-adjoint operator on $L^{2}(VX)$. For a
finite graph with $N$ vertices we label the eigenvalues of $\Delta$ as
follows: $0=\lambda_{0}\leq\lambda_{1}\leq...\leq\lambda_{N-1}\leq2(q+1)$.

The heat kernel $K_{X}(t,x,y):\mathbb{R}_{\geq0}\times X\times X\rightarrow\mathbb{R}$
on $X$ is the solution of
\begin{subequations}
\begin{align}\label{eq:heat_equation_a}
\Delta K_{X}(t,x_{0},x)+\frac{\partial}{\partial t}K_{X}(t,x_{0},x)
& =0\\ \label{eq:heat_equation_b}
K_{X}(0,x_{0},x) & = \begin{cases}
1 & \mbox{if\ } x=x_0  \\
0 & \mbox{otherwise.}
\end{cases}
\end{align}
\end{subequations}
 Sometimes we write  $K_{X}(t,x_0,x)=K_{X}(t,x)$ when a base point $x_{0}$ is
understood.

Whenever $X$ is a countable graph with bounded vertex degree, the
heat kernel of $X$ exists and is unique among bounded functions (\cite{Dodziuk,Dodziuk06}).
Let $X$ be a regular covering of $Y$ via the map $\pi$ and group $\Gamma$, then
it is formally immediate that
\[
K_{Y}(t,y_{0},y)=\sum_{x\in\pi^{-1}(y)}K_{X}(t,x_{0},x)
\]
where $y_{0}=\pi(x_{0})$; alternatively, with a chosen $x$ such
that $y=\pi(x)$, we have that
\[
K_{Y}(t,y_{0},y)=\sum_{\gamma\in\Gamma}K_{X}(t,x_{0},\gamma x).
\]
One can show that the heat kernel decays sufficiently rapidly so
that the above equalities are not only formal, but indeed are convergent series.
We refer to \cite{CJK10, Dodziuk} as well as
the heat kernel formula for trees in Proposition
\ref{pro:The-heat-kernel} and the bounds in subsection
\ref{sub:Universal-bounds-for}.

Consider the numbers $a_{n}(x)$ defined by\[
e^{(q+1)t}K_{X}(t,x)=\sum_{n=0}^{\infty}a_{n}(x)\frac{t^{n}}{n!}\]
then it is well-known and simple to see that $a_{n}(x)$ is the number
of paths from $x_{0}$ to $x$ as defined above.

\subsection{The $I$-Bessel function}

Classically, the $I$-Bessel function $I_{x}(t)$ is defined as a certain
solution to the differential equation
\[
t^{2}\frac{d^{2}w}{dt^{2}}+t\frac{dw}{dt}-(t^{2}+x^{2})=0.
\]
 For integer values of $x$, it is immediately shown that $I_{x}=I_{-x}$
and, for positive integer values of $x$, we have the series representation
\begin{equation}
I_{x}(t)=\sum\limits _{n=0}^{\infty}\frac{(t/2)^{2n+x}}{n!\,\Gamma(n+1+x)}\label{Iseries}\end{equation}
 as well as the integral representation \begin{equation}
I_{x}(t)=\frac{1}{\pi}\int\limits _{0}^{\pi}e^{t\cos(\theta)}\cos(\theta x)d\theta.\label{Iintegral}\end{equation}
 The mathematical literature contains a vast number of articles and
monographs which study the many fascinating properties and manifestations
of the $I$-Bessel functions, as well as other Bessel functions. The
connection with the discrete heat equation comes from the basic
relation
\begin{equation}
  \label{eq:IBessel_relation}
I_{x+1}(t)+I_{x-1}(t)=2\frac{d}{dt}I_{x}(t),
\end{equation}
 which easily can be derived from the integral representation and
trigonometric identities.  This relation will be used in the proof of
Proposition \ref{pro:The-heat-kernel}.

\subsection{Universal bounds for the $I$-Bessel function\label{sub:Universal-bounds-for}}

We have the following uniform bounds from \cite{CJK10} which used
\cite{Paltsev99}. For any $t>0$ and integer $x\geq0$, we have that
\[
\sqrt{t}\cdot e^{-t}I_{x}(t)\leq\left(\frac{t}{t+x}\right)^{x/2}=\left(1+\frac{x}{t}\right)^{-x/2}.
\]
As stated, the above bound is enough to show that the periodization procedure in our
setting gives rise to convergent sum expressions for the heat kernel.

\subsection{An integral transform of I-Bessel\label{sub:An-integral-transform}}

For integers $n$ and $s\in\mathbf{C}$ with $\textrm{Re}(s)\neq0$
we have e.g. from \cite{Oberttinger}, that \[
\int\limits _{0}^{\infty}e^{-st}e^{-t}I_{n}(t)dt=\frac{\left(s+1-\sqrt{(s^{2}+2s)}\right)^{n}}{\sqrt{(s^{2}+2s)}}.\]

We will consider the transform, essentially the Laplace transform,\[
Gf(u)=(u^{-2}-q)\int_{0}^{\infty}e^{-(qu+1/u)t}e^{(q+1)t}f(t)dt.\]
In view of the above formula, applying the transform to the heat kernel
building block, we get
\[
G\left(e^{-(q+1)t}q^{-k/2}I_{k}(2\sqrt{q}t)\right)(u)=u^{k-1}
\]
for $k\geq0$ and $u>0.$

\section{Heat kernels on regular graphs}

\subsection{A heat kernel expression for regular trees \label{sub:heatontree}}

Let $X$ be the $(q+1)$-regular tree and $x_{0}\in X$ a base point.
From its characterizing properties (\ref{eq:heat_equation_a}) and
(\ref{eq:heat_equation_b}), it is immediate to show that the heat
kernel on a graph is invariant with respect to any graph
automorphism $g$:
$$
K(t,gx_0,gx)=K(t,x_0,x).
$$
In particular, we have that the heat kernel $K(t,x_{0},x)\ $ on the
tree $X$ is radial, that is it depends only on
$r=d(x_{0},x)$. Therefore we can write the heat kernel as
$K(t,r)$. Expressions for $K(t,r)$ were established by Bednarchak
\cite{Bednarchak97}, Chung and Yau \cite{ChungYau99}, Cowling, Meda
and Setti \cite{Cowling00}, as well as Horton, Newland and Terras
\cite{HNT06}.  In the physics literature regular trees are called
Bethe lattices.  As stated, one of the main results of this paper is a
formula for the heat kernel on $X$, which we consider to be new since
we did not find the expression in either the mathematical or physics
literature.

As an example of a known expression, Chung and Yau \cite{ChungYau99}
prove that in the radial coordinate $r$, the heat kernel of the
$(q+1)$-regular tree is given by
\begin{subequations}
  \begin{align}\label{eq:CYa}
    K(t,r)& =\frac{2e^{-(q+1)t}}{\pi q^{r/2-1}}\int_{0}^{\pi}
    \frac{\exp\left(2t\sqrt{q}\cos u\right)
\sin u(q\sin(r+1)u-\sin(r-1)u)}{(q+1)^{2}-4q\cos^{2}u}du\\ \nonumber
\mbox{for $r>0$, and }\\ \label{eq:CYb}
K(t,0)& =\frac{2q(q+1)e^{-(q+1)t}}{\pi}\int_{0}^{\pi}
\frac{\exp\left(2t\sqrt{q}\cos u\right)\sin^{2}u}{(q+1)^{2}-4q\cos^{2}u}du.
  \end{align}
\end{subequations}

The formula which we prove is given in the following proposition.

\begin{prop}
\label{pro:The-heat-kernel}The heat kernel of the $(q+1)$-regular
tree is given in the radial coordinate $r\geq0$ as\[
K(t,r)=q^{-r/2}e^{-(q+1)t}I_{r}(2\sqrt{q}t)-(q-1)\sum_{j=1}^{\infty}q^{-(r+2j)/2}e^{-(q+1)t}I_{r+2j}(2\sqrt{q}t),\]
where $I$ denotes the $I$-Bessel function.
\end{prop}

\proof It is immediate that when $t=0,$ the above series is equal to
$\delta_{0}(r)$ as required, since $I_{x}(0)=0$ for $x \neq 0$ and
$I_{0}(0)=1.$ Since $K(0,r)=\delta_{0}(r)$, it remains to show that
the above series is equal to $K(t,0)$ for $t>0$.

Denote by $f(r)=K(t,r)$ and $\fdot(r)=\partial K(t,r)/\partial t.$ If
$r=0$, then the differential equation (\ref{eq:heat_equation_a}) for
the heat kernel takes the form
\[
(q+1)f(0)-(q+1)f(1)+\fdot(0)=0,
\]
and for $r>0$ the differential equation becomes
 \[
(q+1)f(r)-qf(r+1)-f(r-1)+\fdot(r)=0.
\]

Let $g(r)$ be the above series expansion times $e^{(q+1)t}$, or, when
written out,
\[
g(r)=q^{-r/2}I_{r}(2\sqrt{q}t)-(q-1)\sum_{j=1}^{\infty}q^{-(r+2j)/2}I_{r+2j}(2\sqrt{q}t)
\]
It is an elementary exercise to show that the series expansion
satisfies the characterizing differential equation for the heat kernel
if and only if we have the differential equations

\begin{equation}\label{requals0}
-(q+1)g(1)+\gdot(0)=0,
\end{equation}
and for every $r>0$
\begin{equation}\label{rnotequals0}
-qg(r+1)-g(r-1)+\gdot(r)=0.
\end{equation}

Let us first verify (\ref{requals0}). We begin by writing the
\begin{align*}
\text{\rm Left-hand-side of (\ref{requals0})} & =-(q+1)q^{-1/2}I_{1}(2\sqrt{q}t)+(q+1)(q-1)\sum_{j=1}^{\infty}q^{-(1+2j)/2}I_{1+2j}(2\sqrt{q}t)\\
 &\qquad\qquad
+2\sqrt{q}I_{0}^{\prime}(2\sqrt{q}t)-(q-1)2\sqrt{q}\sum_{j=1}^{\infty}q^{-2j}I_{2j}^{\prime}(2\sqrt{q}t).
\end{align*}
Using the basic relation $I_{r-1}(z)+I_{r+1}(z)=2I_{r}^{\prime}(z)$
(\ref{eq:IBessel_relation}), we obtain
the expression
\begin{align}\label{requals0formula}
\notag \text{\rm Left-hand-side of (\ref{requals0})} & =-(q+1)q^{-1/2}I_{1}(2\sqrt{q}t)+(q+1)(q-1)\sum_{j=1}^{\infty}q^{-(1+2j)/2}I_{1+2j}(2\sqrt{q}t)\\
\notag &\hskip .50in +\sqrt{q}(I_{1}^{{}}(2\sqrt{q}t)+I_{-1}^{{}}(2\sqrt{q}t))\\
 & \hskip .50in -(q-1)\sqrt{q}\sum_{j=1}^{\infty}q^{-j}(I_{2j+1}^{{}}(2\sqrt{q}t)+I_{2j-1}^{{}}(2\sqrt{q}t)).
\end{align}
Collecting terms, and recalling that $I_{-1}=I_{1}$, we can evaluate the coefficient of
each $I$-Bessel function in (\ref{requals0formula}):

 \begin{align*}
I_{0} & :\text{there are no $I_{0}$ terms,}\\
I_{1} & :-(q+1)q^{-1/2}+2\sqrt{q}-(q-1)\sqrt{q}q^{-1}=0,\\
I_{2j+1} &
:(q^{2}-1)q^{-(1+2j)/2}-(q-1)\sqrt{q}q^{-j}-(q-1)\sqrt{q}q^{-(j+1)} \\
 & \qquad=(q-1)q^{-j}((q+1)q^{-1/2}-q^{1/2}-q^{-1/2})\\
&\qquad=0.
 \end{align*}
In other words, the left-hand-side of (\ref{requals0}) is zero, as required.

Let us now check the case when $r>0$.  Again,  we begin by writing the
\begin{align}\label{rnotequals0formula}
\text{\rm Left-hand-side of (\ref{rnotequals0})} & =-q^{1-(r+1)/2}I_{r+1}(2\sqrt{q}t)+q(q-1)\sum_{j=1}^{\infty}q^{-(r+1+2j)/2}I_{r+1+2j}(2\sqrt{q}t)\\
 &\qquad -q^{-(r-1)/2}I_{r-1}(2\sqrt{q}t)+(q-1)\sum_{j=1}^{\infty}q^{-(r-1+2j)/2}I_{r-1+2j}(2\sqrt{q}t)\notag \\
 & \qquad+2\sqrt{q}q^{-r/2}I_{r}^{\prime}(2\sqrt{q}t)-2\sqrt{q}(q-1)\sum_{j=1}^{\infty}q^{-(r+2j)/2}I_{r+2j}^{\prime}(2\sqrt{q}t)\notag \\
 & =-q^{1/2-r/2}I_{r+1}(2\sqrt{q}t)+q(q-1)\sum_{j=1}^{\infty}q^{-(r+1+2j)/2}I_{r+1+2j}(2\sqrt{q}t)\notag \\
 &\qquad -q^{-(r-1)/2}I_{r-1}(2\sqrt{q}t)+(q-1)\sum_{j=1}^{\infty}q^{-(r-1+2j)/2}I_{r-1+2j}(2\sqrt{q}t\notag )\\
 &\qquad +\sqrt{q}q^{-r/2}(I_{r+1}^{{}}(2\sqrt{q}t)+I_{r-1}^{{}}(2\sqrt{q}t))\notag \\
 &\qquad -\sqrt{q}(q-1)\sum_{j=1}^{\infty}q^{-(r+2j)/2}(I_{r+1+2j}^{{}}(2\sqrt{q}t)+I_{r-1+2j}^{{}}(2\sqrt{q}t)).
 \end{align}
As above, we can evaluate the coefficient of
each $I$-Bessel function in (\ref{rnotequals0formula}):

 \begin{align*}
I_{r-1} & :-q^{-(r-1)/2}+\sqrt{q}q^{-r/2}=0\\
I_{r} & : \text{there are no $I_{r}$ terms,}\\
I_{r+1} & :-q^{1/2-r/2}+(q-1)q^{-(r+1)/2}+\sqrt{q}q^{-r/2}-\sqrt{q}(q-1)q^{-(r+2)/2}=0\\
I_{r+2j+1} & :q(q-1)q^{-(r+1+2j)/2}+(q-1)q^{-(r+1+2j)/2}\\
 & \ \ \ -\sqrt{q}(q-1)q^{-(r+2j)/2}-\sqrt{q}(q-1)q^{-(r+2j+2)/2}\\
 & =(q-1)q^{-r/2}q^{-j}(q^{1/2}+q^{-1/2}-q^{1/2}-q^{-1/2})\\
& =0.
 \end{align*}
In other words, the left-hand-side of (\ref{rnotequals0}) is zero, as required,
which completes the proof of the proposition. $\Box$

In the following subsection we indicate another approach to the proof
of Proposition \ref{pro:The-heat-kernel}.

\subsection{The horospherical transform}

Every geodesic ray $\gamma$ in the tree emanating from a fixed base
point $x_0$ can be viewed as an ``ideal boundary point at
infinity''. To each such $\gamma$ there are associated horospheres,
one for each integer $n$:
$$
\mathcal{H}_{n}=\{x\in X: \lim_{k\rightarrow\infty}\left[
  d(\gamma(k),x)-k\right]=n\}
$$
where $d$ is the natural combinatorial distance in the graph.

We fix a geodesic ray $\gamma$ and may then consider the associated \emph{horospherical transform} of functions
$f:X\rightarrow\mathbb{R}$ denoted by
\[
\mathbf{H}f:\mathbb{Z\rightarrow R}
\]
and defined by $\mathbf{H}f(n)=\sum_{x\in\mathcal{H}_{n}}f(x).$   For a radial
function, decaying fast enough, we have the inversion formula
\begin{equation}
  \label{eq:inversion_formula}
f(r)=q^{-r}(\mathbf{H}f)(r)-(q-1)\sum_{j=1}^{\infty}q^{-(r+2j)}
(\mathbf{H}f)(r+2j)
\end{equation}
 for $r\geq0$. This is stated in \cite{HNT06} on pages 7-8 for $f$
of finite support.

If we apply the horospherical transform to the equations
 (\ref{eq:heat_equation_a}, \ref{eq:heat_equation_b}) characterizing
 the heat kernel, we get
\[
(q+1)\mathbf{H}K(t,n)-\left(q\mathbf{H}K(t,n+1)
+\mathbf{H}K(t,n-1)\right)+\frac{\partial}{\partial t}\mathbf{H}K(t,n)=0,\]
 for $n\in\mathbb{Z}$ and with $\mathbf{H}K(0,n)=\delta_{0}(n).$
The solution to this difference-differential equation can be seen
to be (cf. section 3.2 in \cite{HNT06} or \cite{KaNe06})\[
f(t,n)=q^{-n/2}e^{-(q+1)t}I_{n}\left(2\sqrt{q}t\right).\]
 As already remarked, the heat kernel on a regular tree is radial,
so by inserting the above expression into the inversion formula
(\ref{eq:inversion_formula}) (here
we need to go beyond finitely supported functions for which this formula
was stated in \cite{HNT06}) we can get a different proof of Proposition
\ref{pro:The-heat-kernel}.

\subsection{Heat kernels on regular graphs}

Let $q>0$ be an integer and $X$ a $(q+1)$-regular graph. We fix
a base point $x_{0}\in VX$ which we will suppress in the notation.
The heat kernel on $X$ can be obtained from periodizing the heat
kernel $K_{q+1}$ on the universal covering space, the $(q+1)$-regular
tree $T_{q+1}$, over the covering group $\Gamma.$ Following the
remarks in subsection \ref{sub:The-combinatorial-Laplacian} we have
(with a slight abuse of notation) that\[
K_{X}(t,x)=\sum_{\gamma\in\Gamma}K_{q+1}(t,\gamma x).\]
Recall that $c_{n}(x)$ denotes the number of paths without backtracking from
the identity $x_{0}$ to $x$ of length $n$ in $X$. We also use
$c_{n}^{0}=c_{n}(x_{0})$ as notation for the number of geodesic loops,
i.e. closed paths without backtracking, starting at $x_{0}$. Note
that $c_{n}(x)$ is equal to the number of elements of the form $\gamma x$
for some $\gamma\in$$\Gamma$ on the radius $n$ sphere in $T_{q+1}$.
We therefore have$ $\[
K_{X}(t,x)=\sum_{n\geq0}c_{n}(x)K_{q+1}(t,n)\]
or more explicitly by inserting the expression from Proposition \ref{pro:The-heat-kernel} for $K_{q+1}(t,n)$,
\[
K_{X}(t,x)=e^{-(q+1)}\sum_{n\geq0}c_{n}(x)\sum_{j=0}^{\infty}d_{q}(j)q^{-n/2-j}I_{n+2j}(2\sqrt{q}t),\]
where $d_{q}(j)$ is $1$ if $j=0$ and $1-q$ otherwise. A rearrangement
of the terms gives \[
K_{X}(t,x)=e^{-(q+1)t}\sum_{m\geq0}b_{m}(x)q^{-m/2}I_{m}(2\sqrt{q}t),\]
where $b_{m}(x)=c_{m}(x)-(q-1)(c_{m-2}(x)+c_{m-4}(x)+...)$ where the
last term is $c_{1}(x)$ if $m$ is odd and $c_{0}(x)$ if $m$ is even.
This is also with the understanding that $b_{0}(x)=c_{0}(x)$ and
$b_{1}(x)=c_{1}(x)$. This proves Theorem \ref{thm:The-heat-kernel}.

Now specialize to $x=x_0.$ In view of Propositions
\ref{pro:pathcombinatorics} and \ref{pro:The-heat-kernel} we obtain
Corollary \ref{corollary} as well.

\subsection{Spectral theory}

An excellent reference here is that of Mohar and Woess \cite{MoWo89}.
One has that there are spectral measures $\mu_{x}$ such that (suppressing
$x_{0}$)
\begin{equation}
  \label{eq:spectral_measure_continuous}
K_{X}(t,x)=\int e^{-\lambda t}d\mu_{x}(\lambda).
\end{equation}

In particular if $X$ is a finite graph with $n$ vertices, then the
Laplacian has eigenvalues $0=\lambda_{0}<\lambda_{1}\leq...\leq\lambda_{n-1}$
and corresponding orthonormal eigenfunctions $\phi_{j}$. The heat
kernel may thus be written as
\begin{equation}
  \label{spectral_measure_finite}
K_{X}(t,x_{0},x)=\frac{1}{n}
\sum_{j=0}^{n-1}e^{-\lambda_{j}t}\phi_{j}(x)\overline{\phi_{j}(x_{0})}.
\end{equation}

\section{Ihara formulas}

In this final section we describe how the $G$-transform introduced in
Section \ref{sub:An-integral-transform} applied to the heat kernel
gives rise to the Ihara zeta function.

\subsection{Zeta functions}

Motivated by Selberg's work, Ihara defined a zeta function for a finite graph $X$,
which is now referred to as \emph{the Ihara zeta function}.  The product formula
for the Ihara zeta function is
\[
\zeta_{X}^{Ih}(u)=\prod_{\left[P\right]}(1-u^{l(P)})^{-1}
\]
where the product is over equivalence classes of prime geodesics and
$l$ the length. (Actually, Ihara worked in a specific group setting,
but Serre remarked in the preface of \cite{Serre80} that the definition
could be given a simple interpretation in terms of graphs.) By a general
calculation, see for example \cite[p. 29]{Terras11}, one has\[
\log\zeta_{X}^{Ih}(u)=\sum_{m=1}^{\infty}\frac{N_m}{m}u^{m},\]
where $N_m$ is the number of closed geodesics of length $m.$ Thus
the Ihara function is a zeta type function similar to those appearing
in the classical works of Artin, Hasse, and Weil on counting points
of varieties in finite fields.

The numbers $N_m$ are not defined for infinite graphs $X$. In the
case of transitive graphs, a natural replacement is $N_{m}^{0},$ since
in the finite case one has $N_m=nN_{m}^{0}$, where $n$ is the number
of vertices. So we can define a zeta function for any (not necessarily
finite) vertex transitive graph via\[
\log\zeta_{X}(u)=\sum_{m=1}^{\infty}\frac{N_{m}^{0}}{m}u^{m}.\]
More exotically one could define a two variable function $\zeta$
via\[
\log\zeta_{X}(u,x)=\sum_{m=1}^{\infty}\frac{b_{m}(x)}{m}u^{m}.\]
Following Riemann one has another set of zeta functions by instead
taking the Mellin transform of the the heat kernel. This is a subject
for another paper.

\subsection{The periodization side}

We will apply the transform
\[
Gf(u)=(u^{-2}-q)\int_{0}^{\infty}e^{-(qu+1/u)t}e^{(q+1)t}f(t)dt
\]
 first to our heat kernel expression
 \[
K_{X}(t,x_{0},x)=
e^{-(q+1)t}\sum_{m=0}^{\infty}b_{m}(x)q^{-m/2}I_{m}(2\sqrt{q}t)
\]
of Theorem \ref{thm:The-heat-kernel}.
Using the basic formula in subsection \ref{sub:An-integral-transform},
in the case $x\neq x_{0}$, so $b_{0}(x)=0$ we have that the $G$-transform
of the heat kernel is equal to
\[
\frac{1}{u}\sum_{m=0}^{\infty}b_{m}(x)u^{m}=\frac{\partial}{\partial u}\sum_{m=1}^{\infty}\frac{b_{m}(x)}{m}u^{m}
=\frac{\text{\ensuremath{\partial}}}{\partial u}\log\zeta_X(u,x).
\]
In the case $x=x_{0}$, the $G$-transform of the
heat kernel on the diagonal is equal to
\[
\frac{1}{u}\sum_{m=0}^{\infty}b_{m}u^{m}=\frac{\partial}{\partial u}\sum_{m=1}^{\infty}\frac{b_{m}}{m}u^{m}+\frac{\partial}{\partial u}\log u.
\]

In the vertex transitive case, for the case $x=x_{0}$ we apply the
transform to the expression of Corollary \ref{corollary}
\[
K_{X}(t,x_{0},x_{0})=
K_{q+1}(t,x_{0})+
e^{-(q+1)t}\sum_{m=1}^{\infty}N_{m}^{0}q^{-m/2}I_{m}(2\sqrt{q}t)\]
which gives (summing the geometric series arising from the first term)\[
1/u-(q-1)\frac{u}{1-u^{2}}+\frac{1}{u}\sum_{m=1}^{\infty}N_{m}^{0}u^{m}.
\]
This can in turn be written as
\[
\frac{\partial}{\partial u}\log u+\frac{q-1}{2}
\frac{\partial}{\partial u}\log(1-u^{2})+\frac{\partial}{\partial u}
\sum_{m=1}^{\infty}\frac{N_{m}^{0}}{m}u^{m}\]
\[
\qquad
=\frac{\partial}{\partial u}\log u+\frac{q-1}{2}
\frac{\partial}{\partial u}
\log(1-u^{2})+\frac{\partial}{\partial u}\log(\zeta_X(u)).
\]

We have proven

\begin{prop}
  \label{prop:G_transform_hk}
For $X$ a $(q+1)$-regular vertex transitive graph,
$$(GK_X)(\cdot, x_0)(u)=
\frac{\partial}{\partial u}\left[
\log u+\frac{q-1}{2}\log(1-u^{2})+\log\zeta_X(u).
\right]
$$
\end{prop}

In summary, the $G$-transform of the heat kernel yields expressions
involving the Ihara zeta function together with other trivial terms.
In the setting of compact quotients of rank one symmetric spaces,
there is a similar change of variables in the Laplace transform so
that when applied to the trace of the heat kernel, one obtains the
Selberg zeta function.  With this in mind, our approach to the Ihara
zeta function as an integral transform of the heat kernel is in line
with a known approach to the Selberg zeta function in many settings.

\subsection{Ihara's determinantal formula\label{sub:Ihara's-determinantal-formula}}

Now we deduce the classical Ihara determinantal formula. Let $X$
be a $(q+1)$-regular graph of finite vertex cardinality $n$. From
(\ref{spectral_measure_finite}),
\[
K_{X}(t,x_{0})=\frac{1}{n}\sum_{j=0}^{n-1}e^{-\lambda_{j}t}.
\]
The $G$-transform of the righthand side is a simple integration which yields
\[
\frac{1}{n}(u^{-2}-q)\sum_{j=0}^{n-1}\frac{1}{qu+1/u-(q+1-\lambda_{j})}=
-\frac{1}{n}\frac{\partial}{\partial u}
\sum_{j=0}^{n-1}\log\frac{1}{u}\left(1-(q+1-\lambda_{j})u+qu^{2}\right).
\]
Comparing this last expression with Proposition
\ref{prop:G_transform_hk}
from the periodization side (and also verifying that
integration constants match up) we immediately get Ihara's formula, namely
\[
\frac{1}{\zeta_X^{Ih}(u)}=(1-u^{2})^{n(q-1)/2}\det((1-(q+1)u+qu^{2})I+\Delta u),
\]
since, as remarked above,
\[
\frac{1}{u}\sum_{m=1}^{\infty}N_mu^{m}=
\frac{\partial}{\partial u}\log\zeta_X^{Ih}(u).
\]
This formula is also known to hold more generally
for non-regular graphs, see the references mentioned in the introduction.

\subsection{First extension of Ihara's formula}

Using spectral theory, we obtain a similar formula for infinite
transitive graphs, this time using our zeta function instead of
Ihara's.  The common point is that both zeta functions are obtained as
$G$-transforms of the heat kernel, and the determinantal formula
follows from having another expression for the heat kernel, namely
that which comes from spectral theory.

We use the notation $\mu=\mu_{x_{0}}$.  Since the functions involved
are positive, we may change the order of integration in our integral
transforms and arrive at the expression
Equating the expression in Proposition \ref{prop:G_transform_hk} with
the $G$-transform of the spectral expansion of the heat kernel given
in (\ref{eq:spectral_measure_continuous}), we arrive at
\[
\frac{\partial}{\partial u}\left[
\log u+\frac{q-1}{2}\log(1-u^{2})+\log\zeta_X(u)
\right]
=-\frac{\partial}{\partial u}\int\log\frac{1}{u}\left(1-(q+1-\lambda)u+qu^{2}\right)d\mu(\lambda).
\]
Note: since the functions involved
are positive, we are justified in interchanging the spectral integral
with the $G$-transform integral on the righthand side.
We now integrate this equality, noting that at $u=0$ both sides are $0$
to determine the integration constants. We get the formula
\[
\log u+\frac{q-1}{2}\log(1-u^{2})+\log\zeta_X(u)
=-\int\log\frac{1}{u}\left(1-(q+1-\lambda)u+qu^{2}\right)d\mu(\lambda),\]
which leads to
\begin{equation}\label{Iharafirst}
\zeta_{X}(u)^{-1}=(1-u^{2})^{(q-1)/2}\exp
\left[\int\log\left(1-(q+1-\lambda)u+qu^{2}\right)d\mu(\lambda)\right].
\end{equation}
This is Theorem \ref{thm:iharainfinite}.
We further remark that
equation (\ref{Iharafirst}) clearly generalizes the Ihara determinant
formula since for vertex transitive graphs with a finite number $n$
vertices one has $\zeta_X^{Ih}=\zeta_X^{n}$.

\subsection{Second extension of Ihara's formula}

Here we do not specialize to $x=x_{0}$. The resulting identities involve
counting geodesics paths, not only closed geodesics paths.
Alternatively, as in the most classical situation, our consideration corresponds to
computing the Hurwitz zeta function instead of the Riemann zeta function.
With the same calculations as above, one gets in the finite graph case,
(at one point one uses orthogonality of eigenfunctions) the formula
\[
-\log\zeta_X(u,x)=\frac{1}{n}\sum_{j=0}^{n-1}f_{j}(x)\bar{f_{j}}(x_{0})\log(1-(q+1-\lambda_{j})u+qu^{2}).
\]
Thus the eigenfunctions come in to determine the more precise count
of geodesics. The lead asymptotic as the length goes to infinity
behaves the same as for the closed geodesics since the trivial eigenvalue
has the constant function as eigenfunction. From an intuitive viewpoint,
this observation is clear: For a fixed $x$ and large length $m$ the geodesics
do not look much different from a closed geodesic. In symbols, if $m\gg1$
then $x\approx x_{0}$.

We have the analogous formula for infinite regular graphs, namely that
\[
-\log\zeta_X(u,x)=\int\log(1-(q+1-\lambda_{j})u+qu^{2})d\mu_{x}(\lambda).
\]

\def\cprime{$'$}

\vspace{5mm}

\noindent Gautam Chinta \\
 Department of Mathematics \\
 The City College of New York \\
 Convent Avenue at 138th Street \\
 New York, NY 10031 U.S.A. \\
 e-mail: chinta@sci.ccny.cuny.edu

\vspace{5mm}
 Jay Jorgenson \\
 Department of Mathematics \\
 The City College of New York \\
 Convent Avenue at 138th Street \\
 New York, NY 10031 U.S.A. \\
 e-mail: jjorgenson@mindspring.com

\noindent \vspace{5mm}

\noindent Anders Karlsson \\
Mathematics Department \\
University of Geneva \\
1211 Geneva, Switzerland \\
e-mail: anders.karlsson@unige.ch

\begin{thebibliography}{HNT06}

\bibitem[Ahu87]{Ahumada87}
Guido Ahumada.
\newblock Fonctions p\'eriodiques et formule des traces de {S}elberg sur les
  arbres.
\newblock {\em C. R. Acad. Sci. Paris S\'er. I Math.}, 305(16):709--712, 1987.

\bibitem[Bas92]{Bass92}
Hyman Bass.
\newblock The {I}hara-{S}elberg zeta function of a tree lattice.
\newblock {\em Internat. J. Math.}, 3(6):717--797, 1992.

\bibitem[Bed97]{Bednarchak97}
Debe Bednarchak.
\newblock Heat kernel for regular trees.
\newblock In {\em Harmonic functions on trees and buildings ({N}ew {Y}ork,
  1995)}, volume 206 of {\em Contemp. Math.}, pages 111--112. Amer. Math. Soc.,
  Providence, RI, 1997.

\bibitem[CJK10]{CJK10}
Gautam Chinta, Jay Jorgenson, and Anders Karlsson.
\newblock Zeta functions, heat kernels, and spectral asymptotics on
  degenerating families of discrete tori.
\newblock {\em Nagoya Math. J.}, 198:121--172, 2010.

\bibitem[CMS00]{Cowling00}
Michael Cowling, Stefano Meda, and Alberto~G. Setti.
\newblock Estimates for functions of the {L}aplace operator on homogeneous
  trees.
\newblock {\em Trans. Amer. Math. Soc.}, 352(9):4271--4293, 2000.

\bibitem[CMS01]{Bryan02}
Bryan Clair and Shahriar Mokhtari-Sharghi.
\newblock Zeta functions of discrete groups acting on trees.
\newblock {\em J. Algebra}, 237(2):591--620, 2001.

\bibitem[CY99]{ChungYau99}
Fan Chung and S.-T. Yau.
\newblock Coverings, heat kernels and spanning trees.
\newblock {\em Electron. J. Combin.}, 6:Research Paper 12, 21 pp.\
  (electronic), 1999.

\bibitem[DM06]{Dodziuk}
J{\'o}zef Dodziuk and Varghese Mathai.
\newblock Kato's inequality and asymptotic spectral properties for discrete
  magnetic {L}aplacians.
\newblock In {\em The ubiquitous heat kernel}, volume 398 of {\em Contemp.
  Math.}, pages 69--81. Amer. Math. Soc., Providence, RI, 2006.

\bibitem[Dod06]{Dodziuk06}
J{\'o}zef Dodziuk.
\newblock Elliptic operators on infinite graphs.
\newblock In {\em Analysis, geometry and topology of elliptic operators}, pages
  353--368. World Sci. Publ., Hackensack, NJ, 2006.

\bibitem[FZ99]{FoZe99}
Dominique Foata and Doron Zeilberger.
\newblock A combinatorial proof of {B}ass's evaluations of the
  {I}hara-{S}elberg zeta function for graphs.
\newblock {\em Trans. Amer. Math. Soc.}, 351(6):2257--2274, 1999.

\bibitem[GIL08]{Guido08}
Daniele Guido, Tommaso Isola, and Michel~L. Lapidus.
\newblock Ihara's zeta function for periodic graphs and its approximation in
  the amenable case.
\newblock {\em J. Funct. Anal.}, 255(6):1339--1361, 2008.

\bibitem[G{\.Z}04]{Grigorchuk04}
Rostislav~I. Grigorchuk and Andrzej {\.Z}uk.
\newblock The {I}hara zeta function of infinite graphs, the {KNS} spectral
  measure and integrable maps.
\newblock In {\em Random walks and geometry}, pages 141--180. Walter de Gruyter
  GmbH \& Co. KG, Berlin, 2004.

\bibitem[Has89]{Hashimoto89}
Ki-ichiro Hashimoto.
\newblock Zeta functions of finite graphs and representations of {$p$}-adic
  groups.
\newblock In {\em Automorphic forms and geometry of arithmetic varieties},
  volume~15 of {\em Adv. Stud. Pure Math.}, pages 211--280. Academic Press,
  Boston, MA, 1989.

\bibitem[HNT06]{HNT06}
Matthew~D. Horton, Derek~B. Newland, and Audrey~A. Terras.
\newblock The contest between the kernels in the {S}elberg trace formula for
  the {$(q+1)$}-regular tree.
\newblock In {\em The ubiquitous heat kernel}, volume 398 of {\em Contemp.
  Math.}, pages 265--293. Amer. Math. Soc., Providence, RI, 2006.

\bibitem[Iha66]{Ihara66}
Yasutaka Ihara.
\newblock On discrete subgroups of the two by two projective linear group over
  {${\mathfrak p}$}-adic fields.
\newblock {\em J. Math. Soc. Japan}, 18:219--235, 1966.

\bibitem[JL01]{JoLa01}
Jay Jorgenson and Serge Lang.
\newblock The ubiquitous heat kernel.
\newblock In {\em Mathematics unlimited---2001 and beyond}, pages 655--683.
  Springer, Berlin, 2001.

\bibitem[KN06]{KaNe06}
Anders Karlsson and Markus Neuhauser.
\newblock Heat kernels, theta identities, and zeta functions on cyclic groups.
\newblock In {\em Topological and asymptotic aspects of group theory}, volume
  394 of {\em Contemp. Math.}, pages 177--189. Amer. Math. Soc., Providence,
  RI, 2006.

\bibitem[KS00]{KoSu00}
Motoko Kotani and Toshikazu Sunada.
\newblock Zeta functions of finite graphs.
\newblock {\em J. Math. Sci. Univ. Tokyo}, 7(1):7--25, 2000.

\bibitem[McK72]{McKean}
H.~P. McKean.
\newblock Selberg's trace formula as applied to a compact {R}iemann surface.
\newblock {\em Comm. Pure Appl. Math.}, 25:225--246, 1972.

\bibitem[Mn{\"e}07]{Mnev07}
P.~Mn{\"e}v.
\newblock Discrete path integral approach to the {S}elberg trace formula for
  regular graphs.
\newblock {\em Comm. Math. Phys.}, 274(1):233--241, 2007.

\bibitem[MW89]{MoWo89}
Bojan Mohar and Wolfgang Woess.
\newblock A survey on spectra of infinite graphs.
\newblock {\em Bull. London Math. Soc.}, 21(3):209--234, 1989.

\bibitem[OB73]{Oberttinger}
Fritz Oberhettinger and Larry Badii.
\newblock {\em Tables of {L}aplace transforms}.
\newblock Springer-Verlag, New York, 1973.

\bibitem[Pal99]{Paltsev99}
B.~V. Pal{\cprime}tsev.
\newblock On two-sided estimates, uniform with respect to the real argument and
  index, for modified {B}essel functions.
\newblock {\em Mat. Zametki}, 65(5):681--692, 1999.

\bibitem[Ser97]{Serre97}
Jean-Pierre Serre.
\newblock R\'epartition asymptotique des valeurs propres de l'op\'erateur de
  {H}ecke {$T_p$}.
\newblock {\em J. Amer. Math. Soc.}, 10(1):75--102, 1997.

\bibitem[Ser03]{Serre80}
Jean-Pierre Serre.
\newblock {\em Trees}.
\newblock Springer Monographs in Mathematics. Springer-Verlag, Berlin, 2003.
\newblock Translated from the French original by John Stillwell, Corrected 2nd
  printing of the 1980 English translation.

\bibitem[Sun86]{Sun1}
Toshikazu Sunada.
\newblock {$L$}-functions in geometry and some applications.
\newblock In {\em Curvature and topology of {R}iemannian manifolds ({K}atata,
  1985)}, volume 1201 of {\em Lecture Notes in Math.}, pages 266--284.
  Springer, Berlin, 1986.

\bibitem[Sun94]{Sun2}
Toshikazu Sunada.
\newblock Fundamental groups and {L}aplacians [ {MR}0922018 (89d:58128)].
\newblock In {\em Selected papers on number theory, algebraic geometry, and
  differential geometry}, volume 160 of {\em Amer. Math. Soc. Transl. Ser. 2},
  pages 19--32. Amer. Math. Soc., Providence, RI, 1994.

\bibitem[Sun08]{Sun3}
Toshikazu Sunada.
\newblock Discrete geometric analysis.
\newblock In {\em Analysis on graphs and its applications}, volume~77 of {\em
  Proc. Sympos. Pure Math.}, pages 51--83. Amer. Math. Soc., Providence, RI,
  2008.

\bibitem[Ter11]{Terras11}
Audrey Terras.
\newblock {\em Zeta functions of graphs}, volume 128 of {\em Cambridge Studies
  in Advanced Mathematics}.
\newblock Cambridge University Press, Cambridge, 2011.
\newblock A stroll through the garden.

\bibitem[TS07]{ST}
A.~A. Terras and H.~M. Stark.
\newblock Zeta functions of finite graphs and coverings. {III}.
\newblock {\em Adv. Math.}, 208(1):467--489, 2007.

\bibitem[TW03]{Terras03}
Audrey Terras and Dorothy Wallace.
\newblock Selberg's trace formula on the {$k$}-regular tree and applications.
\newblock {\em Int. J. Math. Math. Sci.}, (8):501--526, 2003.

\end{thebibliography}
\end{document}